\documentclass{amsart}
\usepackage[ocgcolorlinks,linktoc=all]{hyperref}
\hypersetup{citecolor=blue,linkcolor=red}
\newtheorem{theorem}{Theorem}[section]
\newtheorem{lemma}[theorem]{Lemma}
\newtheorem{proposition}[theorem]{Proposition}
\newtheorem{corollary}[theorem]{Corollary}
\theoremstyle{definition}

\theoremstyle{remark}
\newtheorem{remark}[theorem]{Remark}

\numberwithin{equation}{section}

\begin{document}

\title[On the Gauss curvature flow]{An application of dual convex bodies to the inverse Gauss curvature flow}
\author[M. N. Ivaki]{Mohammad N. Ivaki}
\address{Department of Mathematics and Statistics,
  Concordia University, Montreal, QC, Canada, H3G 1M8}
\curraddr{}
\email{mivaki@mathstat.concordia.ca}
\thanks{}


\dedicatory{}

\begin{abstract}
By means of dual convex bodies, we obtain regularity of solutions to the expanding Gauss curvature flows with homogeneity degrees $-p$, $0<p<1$. At the end,
we remark that our method can also be used to obtain regularity of solutions to the shrinking Gauss curvature flows with homogeneity degrees less than one.
\end{abstract}

\maketitle
\section{Introduction}
The setting of this paper is the $n$-dimensional Euclidean space, $\mathbb{R}^n.$ A compact convex subset of $\mathbb{R}^n$ with non-empty interior
is called a \emph{convex body}.

Let $K$ be a convex body. The support function of $K$, denoted by $s_K$, is defined as
\begin{align*}
s_K&:\mathbb{S}^{n-1}\to\mathbb{R}\\
 s_{K}(z)&=\max_{y\in \partial K}\langle z,y\rangle,
\end{align*}
where $\langle z,y\rangle$ denotes the standard inner product of $z$ and $y$.

Let $K$ be a strictly convex body, having the origin in its interior such that
its boundary, denoted by $\partial K$, is smoothly embedded in $\mathbb{R}^{n}$ by
\[x_K:\partial K\to\mathbb{R}^{n}.\]
Let $\mathbf{n}_K(x)$ be the outward unit normal vector of $K$ for every $x\in \partial K$. The support function of $K$ has the following
simple form
\[s_{K}(z):= \langle \mathbf{n}_K^{-1}(z), z \rangle,\]
for each $z\in\mathbb{S}^{n-1}$, where $\mathbf{n}_K^{-1}:\mathbb{S}^{n-1}\to \partial K$ is the inverse of the Gauss map $\mathbf{n}_K$. We denote
the standard metric on $\mathbb{S}^{n-1}$ by $\bar{g}_{ij}$ and the standard Levi-Civita connection of $\mathbb{S}^{n-1}$ by $\bar{\nabla}$.
We denote the Gauss curvature of $\partial K$ by $\mathcal{K}$ and remark that, as a function on $\partial K$, it is related to the support function
of the convex body by
\[S_{n-1}:=\frac{1}{\mathcal{K}\circ\mathbf{n}_K^{-1}}:=\det_{\bar{g}}(\bar{\nabla}_i\bar{\nabla}_js+\bar{g}_{ij}s).\]
We denote the matrix of the radii of curvature of $\partial K$ by $\mathfrak{r}=[\mathfrak{r}_{ij}]_{1\leq i,j\leq n-1}$, the entries of $\mathfrak{r}$ are considered as functions on the unit sphere. They are related to the support function by the identity
\[\mathfrak{r}_{ij}:=\bar{\nabla}_i\bar{\nabla}_j s+s\bar{g}_{ij}.\]
We denote the eigenvalues of $[\mathfrak{r}_{ij}]_{1\leq i,j\leq n-1}$ with respect to the metric $\bar{g}_{ij}$ by $\lambda_i$ for $1\leq i\leq n-1,$
and we assume that $\lambda_1\leq\lambda_2\leq\cdots\leq \lambda_{n-1}.$

In this paper, we present a new method to study the long-time behavior of flow by certain powers of the Gauss curvature by linking expanding Gauss curvature flows to shrinking Gauss curvature flows, see section 6 for the latter.

For a given smooth, strictly convex embedding $x_K$, we consider a family of smooth, strictly convex bodies $\{K_t\}_t$, given by the smooth embeddings $x:\partial K\times[0,T)\to \mathbb{R}^{n}$, which
are evolving according to the anisotropic expanding Gauss curvature flow
 \begin{equation}\label{e: flow0}
 \partial_{t}x(\cdot,t):=\frac{\Phi(\mathbf{n}_{K_t}(\cdot))}{\mathcal{K}^{\frac{p}{n-1}}(\cdot,t)}\, \mathbf{n}_{K_t}(\cdot),~~
 x(\cdot,0)=x_{K}(\cdot),
 \end{equation}
where $\Phi:\mathbb{S}^{n-1}\to\mathbb{R}$ is a positive smooth function and $p\in (0,\infty).$
In this equation for each time $t$ we have $x(\partial K,t)=\partial K_t.$
Throughout this paper we assume that $K_0$, the initial smooth convex body, encloses the origin of $\mathbb{R}^n.$

The flow (\ref{e: flow0}) has been studied by Andrews \cite{BA} in dimension two and it was proved that
properly rescaled flows will evolve solutions to the unit circle if $\Phi\equiv1$. For $\Phi\equiv 1$, the inverse Gauss curvature flow has been investigated by Chow and Tsai \cite{BT,BT2}, by Schn\"{u}rer \cite{Shu} for $p=2$ in $\mathbb{R}^3$, and by Q. R. Li \cite{QLI} for $1<p\leq 2$ in $\mathbb{R}^3$. When $\Phi\equiv 1$, the flow (\ref{e: flow0}) is an interesting case of expanding flows $\partial_t x=F^{-\frac{p}{n-1}}\textbf{n},$ where $F$ is a positive, strictly monotone, concave function of principal radii of curvature with the homogeneity degree $n-1$, and $F$ is zero on the boundary of the positive cone $\Gamma=\{(\alpha_1,\cdots,\alpha_{n-1}):\alpha_i>0\}.$ For $p=1$, expanding flows have first been studied by Gerhardt and Urbas \cite{G,JU1,JU}. In particular, they proved that starting the flow with a smooth, strictly convex hypersurface the solution remains smooth and strictly convex. Moreover, the solution exists on $(0,\infty)$ and becomes spherical as it expands. In addition, in \cite{JU} this result was extended to cover the case $0<p<1$. A special class of expanding flows, the inverse mean curvature flow, has been employed by Huisken and Ilmanen to prove the Riemannian Penrose inequality \cite{HI}. Recently, in \cite{G}, Gerhardt thoroughly studied flow of closed star-shaped hypersurfaces and strictly convex hypersurfaces under the expanding flows $\partial_t x=F^{-\frac{p}{n-1}}\textbf{n}$. He demonstrated after a proper rescaling, the rescaled flow will evolve a closed star-shaped hypersurface if $0<p<1$, and a strictly convex hypersurface if $p>1$, to a sphere.

To study the flow (\ref{e: flow0}), one can study the evolution equation of the support functions of $\{K_t\}_t$:
\begin{align}\label{e: flow1}
&s:\mathbb{S}^{n-1}\times[0,T)\to \mathbb{R}\nonumber\\
 &\partial_{t}s(z,t):=\Phi(z) S_{n-1}^{\frac{p}{n-1}}(z,t)=:\Phi F([\mathfrak{r}_{ij}]_{1\leq i,j\leq n-1}),\\
 &s(z,0)=s_{K_0}(z),~~ s(z,t)=s_{K_t}(z)\nonumber.
\end{align}
The short time existence and the uniqueness of solutions to (\ref{e: flow1}) with an initial smooth and strictly convex body follow from the strict parabolicity of the equation, see property \textbf{2} below. Moreover, it is proved in \cite{BT} that if $\Phi\equiv 1$, then the convexity is preserved. Our main theorem, which is stated at the end of this section, bears the restriction $\Phi\equiv 1$, but some of the calculations are for an arbitrary $\Phi.$

We list a few properties of the function $G(\lambda_{1},\lambda_2,\cdots,\lambda_{n-1}):=(\prod_{i=1}^{n-1}\lambda_i)^{\frac{p}{n-1}}$
defined on the positive cone $\Gamma=\{(\alpha_1,\cdots,\alpha_{n-1}):\alpha_i>0\}.$
 \begin{description}
   \item[1] $G$ is a concave positive symmetric real-valued function on $\Gamma,$
   \item[2] $\frac{\partial G}{\partial \lambda_i}(\lambda_{1},\cdots,\lambda_{n-1})>0$, for all $1\leq i\leq n-1,$
   \item[3] $G=0$ on the $\partial \Gamma,$
   \item[4] $G(\lambda_{1},\cdots,\lambda_{n-1})=\frac{1}{G\left(\frac{1}{\lambda_1},\cdots,\frac{1}{\lambda_{n-1}}\right)}$,
   \item[5] $\sum_{i=1}^{n-1}\frac{\partial G}{\partial \lambda_i}\lambda_i=pG.$
 \end{description}
We will use these properties in the remainder of the present text without further mention.

Let $\mathbb{B}_R$ denote the origin-centered ball of $\mathbb{R}^n$ with radius $R>0$. For simplicity, $\mathbb{B}_1$ is denoted by $\mathbb{B}.$ The following is the main result of the paper.
\begin{theorem}[Main Theorem]\label{thm: expanding}
Assume that $n>2$, $\Phi\equiv1$, and $0<p<1$. Let $x_{K}$ be a smooth, strictly convex embedding of $\partial K$. Then there exists a unique solution of (\ref{e: flow0}). The solution remains smooth and strictly convex and the rescaled convex bodies given by $\left(\frac{V(\mathbb{B})}{V(K_t)}\right)^{\frac{1}{n}}K_t$ converge in the $C^{\infty}$ topology to $\mathbb{B}$.
 \end{theorem}
We remark that the $C^1$ convergence and the $C^\infty$ convergence to the unit ball were proved in \cite{BT} and \cite{JU}, respectively. Our main contribution is to give a new proof based by employing dual convex bodies. In the rest of this paper, we will focus on the case $0<p<1.$
\section{evolution equation of the dual convex bodies}
In this section, we calculate the evolution equation of dual convex bodies, which as we will see later on, it can be employed as a useful tool in obtaining regularity of solutions. To my knowledge, evolution equation of dual convex bodies was first introduced by Stancu in \cite{S} in the context of centro-affine normal flows. The method used in \cite{S} for obtaining the evolution equation of dual convex bodies relied on a certain duality property of $L_p$-affine surface areas, therefore that method cannot be used here. In this section, we give an argument that is applicable to flow by powers of the Gauss curvature. It is worth noting that the idea of `duality'(different from the one we consider here) has been used in \cite{BA6,BP}. In \cite{BP}, an upper curvature bound is obtained for the curve shortening flow by studying the isoperimetric profile and to obtain a uniform lower curvature bound, the isoperimetric profile of the `dual' (in this case, the complement of the region enclosed by a simple closed curve) is required. Similarly in \cite{BA6}, an upper curvature bound is obtained from interior non-collapsing and a lower curvature bound is obtained from exterior non-collapsing which is not directly interior non-collapsing of the `dual' (complement) but it is closely related.

Let $K$ be a convex body having the origin in its interior. The dual convex body of $K$ with respect to the origin, denoted by $K^{\circ}$,
is defined as
\[ K^{\circ} = \{ y \in \mathbb{R}^{n} \mid x \cdot y \leq 1,\ \forall x
\in K \}.\]
If $K$ is smooth, then $K^{\circ}$ is also smooth, see \cite{H}.
We furnish all the geometric quantities associated with $K^{\circ}$ by a superscript $~^\circ.$
Next, we find the evolution equation of the support function of $K_t^{\circ}$
as the support function of $K_t$ evolves by (\ref{e: flow1}). To this aim, we parameterize $\partial K_t$ over the unit sphere
\[x_{K_t}=r(z(\cdot,t),t)z(\cdot,t):\mathbb{S}^{n-1}\to\mathbb{R}^n,\]
where $r(z(\cdot,t),t)$ is the radial function of $K_t$ in direction $z(\cdot,t).$ We then recall that how the geometric quantities that we need can be expressed in this parametrization.

Let $K$ be a smooth convex body whose boundary is parameterized over the unit sphere with the radial function $r.$
The metric $[g_{ij}]_{1\leq i,j\leq n-1}$, unit normal $\mathbf{n}$, support function $s$, and the second fundamental form $h_{ij}$ of $\partial K$
can be be written in terms  of $r$ and whose spatial derivatives as follows:

\begin{enumerate}
  \item $\displaystyle g_{ij}=r^2\bar{g}_{ij}+\bar{\nabla}_ir\bar{\nabla}_jr,$
  \item $\displaystyle \mathbf{n}=\frac{1}{\sqrt{r^2+|\bar{\nabla}r|^2}}\left(rz-\bar{\nabla}r\right),$
  \item $\displaystyle s=\langle x,\mathbf{n}\rangle=\frac{r^2}{\sqrt{r^2+|\bar{\nabla}r|^2}},$
  \item $\displaystyle h_{ij}=\frac{1}{\sqrt{r^2+|\bar{\nabla}r|^2}}\left(-r\bar{\nabla}_i\bar{\nabla}_jr+2\bar{\nabla}_ir\bar{\nabla}_jr+
  r^2\bar{g}_{ij}\right).$
\end{enumerate}
A good reference for these identities is \cite{Zhu}.
\begin{lemma}\label{lem: ev r}
As $K_t$ evolve according to the flow (\ref{e: flow1}), then its radial function evolves as follows
\begin{align*}
\partial_t r=\frac{\sqrt{r^2+|\bar{\nabla}r|^2}}{r}\frac{\Phi(\mathbf{n}_{K_t})}{\mathcal{K}^{\frac{p}{n-1}}}.
\end{align*}
\end{lemma}
\begin{proof}
\begin{align*}
\partial_t x&=\partial _t \left(r(z(\cdot,t),t)z(\cdot,t)\right)\\
&=(\partial _t r)z+\langle\bar{\nabla}r,\partial_t z\rangle z+r\partial _t z\\
&=\frac{\Phi(\mathbf{n}_{K_t}(\cdot))}{\mathcal{K}^{\frac{p}{n-1}}}\mathbf{n}_{K_t}\\
&=\frac{\Phi}{\mathcal{K}^{\frac{p}{n-1}}}\frac{1}{\sqrt{r^2+|\bar{\nabla}r|^2}}\left(rz-\bar{\nabla}r\right),
\end{align*}
where from the third line to fourth line we used (2).
By comparing terms on the second line with those on the fourth line, we get
\begin{equation}\label{e: tan com}
r\partial_t z=-\frac{\Phi}{\mathcal{K}^{\frac{p}{n-1}}}\frac{\bar{\nabla}r}{\sqrt{r^2+|\bar{\nabla}r|^2}}
=-\frac{\Phi}{\mathcal{K}^{\frac{p}{n-1}}}\frac{\bar{\nabla}r}{\sqrt{r^2+|\bar{\nabla}r|^2}},
\end{equation}
and
\begin{equation}\label{e: nor com}
\partial _t r+\langle\bar{\nabla}r,\partial_t z\rangle =\frac{\Phi}{\mathcal{K}^{\frac{p}{n-1}}}\frac{r}{\sqrt{r^2+|\bar{\nabla}r|^2}}.
\end{equation}
Replacing $\partial_t z$ in (\ref{e: nor com}) by its equivalent expression from equation (\ref{e: tan com}) completes the proof.
\end{proof}
As $\displaystyle\frac{1}{r}$ is the support function of $\partial K^{\circ}$ (see, \cite{schneider}), we can find the entries of the matrix $[\mathfrak{r}^{\circ}_{ij}]_{1\leq i,j\leq n-1}$ of $\partial K^{\circ}$:
\[\mathfrak{r}^{\circ}_{ij}=\bar{\nabla}_i\bar{\nabla}_j\frac{1}{r}+\frac{1}{r}\bar{g}_{ij}=
\frac{1}{r^3}\left(-r\bar{\nabla}^2_{ij}r+2\bar{\nabla}_ir\bar{\nabla}_jr+r^2\bar{g}_{ij}\right).\]
Thus, we have
\begin{equation}\label{e: h and h dual}
\mathfrak{r}^{\circ}_{ij}=\frac{\sqrt{r^2+|\bar{\nabla}r|^2}}{r^3}h_{ij}.
\end{equation}
We will use this identity in the proof of the next theorem.
\begin{theorem}[The dual evolution equation]
As $\{K_t\}_t$ evolves according to evolution equation (\ref{e: flow1}), then $\{K^{\circ}_t\}_t$ evolves according to
\begin{align}\label{e: ev of dual}
\partial_t s^{\circ}(z,t)=-\Phi\left(\frac{s^{\circ}z+\bar{\nabla}s^{\circ}}{\sqrt{s^{\circ~2}+|\bar{\nabla}s^{\circ}|^2}}\right)
\left(\frac{\left(s^{\circ~2}+ |\bar{\nabla}s^{\circ}|^2\right)^{\frac{(n+1)p}{2(n-1)}+\frac{1}{2}}}{s^{\circ ~ \frac{(n+1)p}{n-1}-1}}\right)\left(S_{n-1}^{\circ}\right)^{-\frac{p}{n-1}}.
\end{align}
\end{theorem}
\begin{proof}
By means of the identities
\[\mathcal{K}=\frac{\det h_{ij}}{\det g_{ij}},~~
 \frac{1}{S_{n-1}^{\circ}}=\frac{\det \bar{g}_{ij}}{\det \mathfrak{r}^{\circ}_{ij}},~~
\frac{\det \bar{g}_{ij}}{\det g_{ij}}=\frac{1}{r^{2n-4}(r^2+|\bar{\nabla}r|^2)},\]
 equation (3), equation (\ref{e: h and h dual}), and Lemma \ref{lem: ev r} calculate
\begin{align*}
\partial_t s^{\circ}(z,t)&=\partial_t \frac{1}{r(z)}\\
&=-\frac{\sqrt{r^2+|\bar{\nabla}r|^2}}{r^3}
\frac{\Phi(\mathbf{n}_{K_t})}{\mathcal{K}^{\frac{p}{n-1}}}\\
&=-\Phi\frac{\sqrt{r^2+|\bar{\nabla}r|^2}}{r^3}\left(\frac{\det g_{ij}}{\det h_{ij}}\right)^{\frac{p}{n-1}}\\
&=-\Phi\frac{\sqrt{r^2+|\bar{\nabla}r|^2}}{r^3}\left(\frac{\det \bar{g}_{ij}}{\det \mathfrak{r}^{\circ}_{ij}}\right)^{\frac{p}{n-1}}
\left(\frac{\det g_{ij}}{\det \bar{g}_{ij}}\right)^{\frac{p}{n-1}}
\left(\frac{\det \mathfrak{r}^{\circ}_{ij}}{\det h_{ij} }\right)^{\frac{p}{n-1}}\\
&=-\Phi\cdot\left(\frac{\sqrt{r^2+|\bar{\nabla}r|^2}}{r^3}\right)^{p+1}\left(\frac{1}{S_{n-1}^{\circ}}\right)^{\frac{p}{n-1}}
\left(r^{2n-4}(r^2+|\bar{\nabla}r|^2)\right)^{\frac{p}{n-1}}\\
&=-\Phi(\mathbf{n}_{K_t})\left(r^2+|\bar{\nabla}r|^2\right)^{\frac{p+1}{2}+\frac{p}{n-1}}r^{-\frac{(n+1)p}{n-1}-3}\left(S_{n-1}^{\circ}\right)^{-\frac{p}{n-1}}\\
&=-\Phi\left(\frac{s^{\circ}z+\bar{\nabla}s^{\circ}}{\sqrt{s^{\circ~2}+|\bar{\nabla}s^{\circ}|^2}}\right)
\left(\frac{\left(s^{\circ~2}+ |\bar{\nabla}s^{\circ}|^2\right)^{\frac{(n+1)p}{2(n-1)}+\frac{1}{2}}}{s^{\circ~ \frac{(n+1)p}{n-1}-1}}\right)\left(S_{n-1}^{\circ}\right)^{-\frac{p}{n-1}},
\end{align*}
where on the last line, we replaced $\displaystyle r$ by $\displaystyle\frac{1}{s^{\circ}}$.
\end{proof}
 \section{Bounding the isoperimetric ratio}
In this section, we assume that $\Phi\equiv 1.$ We recall the following result of Chow and Gulliver from \cite{BG}.
\begin{proposition}\label{prop: chow}\cite{BG} There exists a positive constant $C$ depending only on the initial condition $s(\cdot,0)$ such that, as long as the solution to (\ref{e: flow1}) exists and $\mathfrak{r}_{ij}>0$ on $\mathbb{S}^{n-1}\times [0,T),$ the support function $s(\cdot,t)$ satisfies the uniform gradient estimate $||\bar{\nabla} s(\cdot,t)||\leq C,$
$\mathbb{S}^{n-1}\times [0,T).$ In particular, since the diameter of $~\mathbb{S}^{n-1}$ is $\pi$, we have
\[\max\limits_{\mathbb{S}^{n-1}} s(\cdot,t)-\min\limits_{\mathbb{S}^{n-1}} s(\cdot,t)\leq C\pi\]
for all $t\in[0,T).$
\end{proposition}
Next is an immediate corollary of this proposition.
\begin{corollary}\label{cor: 1}
Assume that $\{K_t\}_{[0,T)}$ is a smooth, strictly convex solution of equation (\ref{e: flow1}) such that $\lim\limits_{t\to T}\min\limits_{\mathbb{S}^{n-1}} s(\cdot,t)=\infty.$ Then we have
\[\lim_{t\to T}\frac{\max\limits_{\mathbb{S}^{n-1}} s(\cdot,t)}{\min\limits_{\mathbb{S}^{n-1}} s(\cdot,t)}=1,~~\frac{\max\limits_{\mathbb{S}^{n-1}} s(\cdot,t)}{\min\limits_{\mathbb{S}^{n-1}} s(\cdot,t)}\leq C\]
for some positive constant $C.$
\end{corollary}
\begin{corollary}\label{cor: 2}
Assume that $\{K_t\}_{[0,T)}$ is a smooth, strictly convex solution of equation (\ref{e: flow1}) such that $\lim\limits_{t\to T}\min\limits_{\mathbb{S}^{n-1}} s(\cdot,t)=\infty.$ Then there is a finite positive constant $C$ such that, the family of convex bodies $\left\{\tilde{K}_t:=\left(\frac{V(\mathbb{B})}{V(K_t)}\right)^{\frac{1}{n}}K_t\right\}_{[0,T)}$ satisfies
\[\frac{1}{C}\leq s_{\tilde{K}_t}\leq C\]
for all time.
\end{corollary}
\begin{proof}
The origin belongs to the interior of $ K_0$, so does to the interior of $K_t$ and $\tilde{K}_t$ as $K_0\subseteq K_t$. Since $V(\tilde{K}_t)=\omega_n,$ we must have $\min\limits_{\mathbb{S}^{n-1}} s_{\tilde{K}_{t}}\leq 1\leq\max\limits_{\mathbb{S}^{n-1}} s_{\tilde{K}_{t}}.$ On the other hand, the ratio $(\max\limits_{\mathbb{S}^{n-1}} s/\min\limits_{\mathbb{S}^{n-1}} s)$ is scaling-invariant, so from Corollary \ref{cor: 1} it follows that $(\max\limits_{\mathbb{S}^{n-1}} s_{\tilde{K}_t}/\min\limits_{\mathbb{S}^{n-1}} s_{\tilde{K}_t})\leq C.$ Therefore, $\frac{1}{C}\leq s_{\tilde{K}_t}\leq C.$
\end{proof}
\section{Upper and lower bounds on the principal curvatures}\label{sec: up and low}
In this section, we present a technique to obtain lower and upper bounds on the principal curvatures of the evolving convex bodies supposing that there are uniform lower and upper bounds on the evolving support functions. Our approach to derive lower and upper bounds consists of three steps. We first obtain an upper bound on the Gauss curvature using Tso's trick \cite{Tso}. We then continue to derive a lower bound on the Gauss curvature using the evolution equation of dual convex bodies (\ref{e: ev of dual}). In the last step, we apply the parabolic maximum principle to the evolution equation of $\mathfrak{r}^{ij}$; to be defined later in this section.
 \begin{lemma}[Upper bound on the Gauss curvature]\label{lem: upper G}
Assume that $\{K_t\}_{[0,t_0]}$ is a smooth, strictly convex solution of equation (\ref{e: flow1}) with $0<R_{-}\leq s_{K_t}\leq R_{+}< \infty$ for $t\in[0,t_0]$. Then
\[\mathcal{K}\leq C't^{\frac{n-1}{p-1}},\]
where $C$ is a constant depending on $p,R_{-},R_{+}$, and $\Phi.$
\end{lemma}
\begin{proof}
Define $\beta=\frac{p}{n-1}.$ We consider the function
\[\Psi=\frac{\Phi S_{n-1}^{\beta}}{2R_{+}-s}.\]
Differentiating $\Phi$ at the point where the minimum of $\Psi$ occurs yields:
\[0=\bar{\nabla}_i\Psi=\bar{\nabla}_i \left(\frac{\Phi S_{n-1}^{\beta}}{2R_{+}-s}\right)\ \ \ {\hbox{and}}\ \ \  \bar{\nabla}_i\bar{\nabla}_j \Psi\geq 0,\]
hence we obtain
\[ \frac{\bar{\nabla}_i (\Phi S_{n-1}^{\beta})}{2R_{+}-s}=-\frac{\Phi S_{n-1}^{\beta} \bar{\nabla}_i s}{(2R_{+}-s)^2}, \] and
\begin{equation}\label{e: tso dual}
\bar{\nabla}_i\bar{\nabla}_j (\Phi S_{n-1}^{\beta})+\bar{g}_{ij} (\Phi S_{n-1}^{\beta})\geq
\frac{-\Phi S_{n-1}^{\beta}\mathfrak{r}_{ij}+2R_{+}\Phi S_{n-1}^{\beta}\bar{g}_{ij}}{2R_{+}-s}.
\end{equation}
When we differentiate $\Phi$ with respect to $t$ we conclude
\begin{align*}
\partial_t\Psi&=\frac{\beta \Phi S_{n-1}^{\beta-1}}{2R_{+}-s}({S}_{n-1})'_{ij}
\left[\bar{\nabla}_i\bar{\nabla}_j\left(\Phi S_{n-1}^{\beta}\right)+\bar{g}_{ij}\left(\Phi S_{n-1}^{\beta}\right)\right]
+\frac{\Phi^2S_{n-1}^{2\beta}}{(2R_{+}-s)^2},
\end{align*}
 where $\displaystyle (S_{n-1})'_{ij} := \frac{\partial S_{n-1}}{\partial  \mathfrak{r}_{ij}}$ is the derivative of the $S_{n-1}$ with respect to the entry $ \mathfrak{r}_{ij}$ of the radii of curvature matrix.
Applying inequality (\ref{e: tso dual}) to the preceding identity we deduce
\begin{equation}\label{e: last step tso dual}
\partial_t\Psi\geq \Psi^2\left(1-(n-1)\beta+2\beta R_{+}\mathcal{H}\right).
\end{equation}
Now, we estimate the mean curvature $\mathcal{H}=\bar{g}_{ij}\mathfrak{r}^{ij}=\sum_{i=1}^{n-1}\frac{1}{\lambda_i}$ from below by a negative power of $\Psi.$ We obtain
\begin{align*}
\mathcal{H}&\geq (n-1)\left(\frac{2R_{+}-s}{\Phi S_{n-1}^{\beta}}\right)^{\frac{1}{(n-1)\beta}}\left(\frac{\Phi}
{2R_{+}-s}\right)^{\frac{1}{(n-1)\beta}}\\
&\geq (n-1)\Psi^{-\frac{1}{(n-1)\beta}} \left(\frac{\min\limits_{\mathbb{S}^{n-1}} \Phi}{2R_{+}-R_{-}}\right)^{\frac{1}{(n-1)\beta}}.
\end{align*}
Consequently, inequality (\ref{e: last step tso dual}) can be rewritten as follows
\begin{align*}
\partial_t\Psi&\geq \Psi^2\left(1-(n-1)\beta+2(n-1)\beta R_{+}\Psi^{-\frac{1}{(n-1)\beta}} \left(\frac{\min\limits_{\mathbb{S}^{n-1}} \Phi}{2R_{+}-R_{-}}\right)^{\frac{1}{(n-1)\beta}}\right)\\
&=\Psi^2\left(C(p)+C'(p,R_{-},R_{+},\Phi)\Psi^{-\frac{1}{p}}\right),
\end{align*} for positive constants $C(p)$ and $C'(p,R_{-},R_{+},\Phi)$. \newline
Hence
\[\partial_t\left(\frac{1}{\Psi}\right)\leq -C'(p,R_{-},R_{+},\Phi)\left(\frac{1}{\Psi}\right)^{\frac{1}{p}}-C(p).\]
From this, it follows that
\[\frac{1}{\Psi}\leq Ct^{\frac{p}{p-1}}\]
for a new constant $C$. Equivalently, we have a bound for $\Psi$ from below. Therefore, we have bounded $\mathcal{K}$ from above in terms of $p,R_{-}, R_{+}$, $\Phi$ and time.
\end{proof}
Before we can obtain a lower bound on the Gauss curvature, we need a preparatory lemma.
\begin{lemma}
Assume that $\{K_t\}_{[0,t_0]}$ is a smooth, strictly convex solution of
\[\partial_t s=-\Phi\left(\frac{sz+\bar{\nabla}s}{\sqrt{s^{2}+|\bar{\nabla}s|^2}}\right)
\left(\frac{\left(s^{2}+ |\bar{\nabla}s|^2\right)^{\frac{(n+1)p}{2(n-1)}+\frac{1}{2}}}{s^{ \frac{(n+1)p}{n-1}-1}}\right)\left(S_{n-1}\right)^{\beta},\]
with $0<R_{-}\leq s_{K_t}\leq R_{+}< \infty$ for $t\in[0,t_0]$.
Here, $\beta$ is a negative real number. Then we have an upper bound on the Gauss curvature depending on $\beta,n,p,R_{-},R_{+}$, and $\Phi.$
\end{lemma}
\begin{proof}
Define $f:= \Phi\left(\frac{sz+\bar{\nabla}s}{\sqrt{s^{~2}+|\bar{\nabla}s|^2}}\right)$, $g:=\left(\frac{\left(s^{2}+ |\bar{\nabla}s|^2\right)^{\frac{(n+1)p}{2(n-1)}+\frac{1}{2}}}{s^{\frac{(n+1)p}{n-1}-1}}\right)$, $Q:=-\partial_t s$, and
consider the function
\[\Psi=\frac{Q}{s-R_{-}/2}.\]
Using the maximum principle, we are going to prove that $\Psi$ is bounded from above by a function of $n, p, R_{-}, R_{+}$, $\Phi$ and time.
At the point where the maximum of $\Psi$ occurs, we have
\[0=\bar{\nabla}_i\Psi=\bar{\nabla}_i \left(\frac{Q}{s-R_{-}/2}\right)\ \ \ {\hbox{and}}\ \ \  \bar{\nabla}_i\bar{\nabla}_j \Psi\leq 0.\]
Hence, we obtain
\begin{equation}\label{e: grad of Q}
\frac{\bar{\nabla}_i Q}{s-R_-/2}=\frac{Q \bar{\nabla}_i s}{(s-R_-/2)^2},
\end{equation}
and consequently,
\begin{equation}\label{e: tso}
\bar{\nabla}_i\bar{\nabla}_jQ+\bar{g}_{ij}Q\leq
\frac{Q\mathfrak{r}_{ij}-R_{-}/2Q
\bar{g}_{ij}}{s-R_{-}/2}.
\end{equation}
We calculate
\begin{align*}
\partial_t\Psi=&-\frac{\beta fgS_{n-1}^{\beta-1}}{s-R_{-}/2}(S_{n-1})'_{ij}
\left[\bar{\nabla}_i\bar{\nabla}_jQ+\bar{g}_{ij}
Q\right]\\
&+\frac{S_{n-1}^{\beta}}{s-R_{-}/2}\partial_t \left(fg\right)+\Psi^2.
\end{align*}
It is not difficult to see that
\[\frac{S_{n-1}^{\beta}}{s-R_{-}/2}\partial_t \left(fg\right)\leq C(n,p,R_-,R_{+},\Phi,f,g)\Psi^2.\]
 Indeed, we only need to take into account that
\[|\partial_t s|=Q,~~ \|\bar{\nabla}\partial_t s\|=\|\bar{\nabla}Q\|=\frac{Q ||\bar{\nabla} s||}{s-R_-/2},~~||\bar{\nabla} s||\leq R_{+},\]
 where for the second equation we used (\ref{e: grad of Q}), and for the third inequality we considered the fact that $||x||^2=s^2+||\bar{\nabla}s||^2\leq R_+^2$ (In fact, $x=sz+\bar{\nabla}s$ and $\langle z,\bar{\nabla}s\rangle=0$, so $||x||^2=s^2+||\bar{\nabla}s||^2$.).
Using this and inequality (\ref{e: tso}) we infer that, at the point where the maximum of $\Psi$ is reached, we have
\begin{equation}\label{e: last step tso}
\partial_t\Psi\leq\Psi^2\left(C'+\frac{\beta R_{-}}{2}\mathcal{H}\right).
\end{equation}
We can control the mean curvature $\mathcal{H}=\sum_{i=1}^{n-1}\frac{1}{\lambda_i}$ from below by a positive power of $\Psi.$ Notice that $\mathcal{H}\geq \frac{n-1}{S_{n-1}^{\frac{1}{n-1}}}.$ This yields
\begin{align*}
\mathcal{H}&\geq (n-1)\left(\frac{s-R_{-}/2}{Q}\right)^{\frac{1}{(n-1)\beta}}
\left(\frac{fg}{s-R_{-}/2}\right)^{\frac{1}{(n-1)\beta}}\\
&\geq (n-1)\Psi^{-\frac{1}{(n-1)\beta}} \left(\frac{C''}{R_{-}-R_{-}/2}\right)^{\frac{1}{(n-1)\beta}}.
\end{align*}
Therefore, we can rewrite the inequality (\ref{e: last step tso}) as follows
\begin{align*}
\partial_t\Psi&\leq \Psi^2\left(C_1-C_2\Psi^{-\frac{1}{(n-1)\beta}} \right)\\
&=-\Psi^2\left(C_2\Psi^{-\frac{1}{(n-1)\beta}}-C_1\right),
\end{align*} for positive constants $C_1$ and $C_2$.
Hence,
\begin{equation}\label{ie: upper on psi}
\Psi\leq \max\left\{C_3,C_4t^{-\frac{(n-1)\beta}{(n-1)\beta-1}}\right\}
\end{equation}
for new constants $C_3$ and $C_4$. The corresponding claim for $\mathcal{K}$ follows.
\end{proof}

\begin{lemma}[Lower bound on the Gauss curvature]\label{lem: lower G}
Assume that $\{K_t\}_{[0,t_0]}$ is a smooth, strictly convex solution of equation (\ref{e: flow1}) with $0<R_{-}\leq s_{K_t}\leq R_{+}< \infty$ for $t\in[0,t_0]$. Then
\[\mathcal{K}\geq \left(\frac{1}{C+C't^{-\frac{p}{1+p}}}\right)^{\frac{n-1}{p}},\]
for constants depending on $n,p,R_{-}$, $R_{+}$, and $\Phi.$
\end{lemma}
\begin{proof}
We apply the previous lemma to the evolution equation (\ref{e: ev of dual}): Observe that as $0<R_{-}\leq s_{K_t}\leq R_{+}< \infty$ we have
\[0<\frac{1}{R_+}\leq s_{K^{\circ}_t}\leq \frac{1}{R_-}< \infty.\]
Now, using (\ref{ie: upper on psi}) with $\beta=-\frac{p}{n-1}$, we obtain that
\[\mathcal{K}^{\circ ~\frac{p}{n-1}}\leq C_3+C_4t^{-\frac{p}{p+1}}.\]
Therefore, we have bounded from above $\mathcal{K}^{\circ}$ in terms of $n,p,R_{-}, R_{+}, \Phi$, and time. To complete the proof, we recall Kaltenbach's identity: for every $x \in \partial K$, there exists an $x^\circ \in \partial K^\circ$ such that
\[\left(\frac{\mathcal{K}}{s^{n+1}}\right)(x)\left(\frac{\mathcal{K}^{\circ}}{s^{\circ n+1}}\right)(x^{\circ})=1,\]
where $x$ and $x^{\circ}$ are related by $\langle x, x^{\circ}\rangle=1$ (A proof of this identity can be found in \cite{Kal}, see also \cite{H,LR}.).\newline
By the above identity, we conclude that $\mathcal{K}$ is bounded from below by constants depending on $n,p,R_{-}, R_{+},\Phi$, and time.
\end{proof}

Denote the inverse matrix of $[\mathfrak{r}_{ij}]_{1\leq i,j\leq n-1}$ by $[\mathfrak{r}^{ij}]_{1\leq i,j\leq n-1}.$ We recall the following evolution equation of $\mathfrak{r}^{ij}$ from \cite{BT} as $\{K_t\}_t$ evolves under (\ref{e: flow1}).
\begin{align*}
 \partial_t \mathfrak{r}^{ij}=&\Phi F'_{kl}\bar{\nabla}_k\bar{\nabla}_l \mathfrak{r}^{ij}-\Phi(F+F'_{kl}\mathfrak{r}_{kl})
 \mathfrak{r}^{ip}\mathfrak{r}^{jp}
+\Phi(tr F')\mathfrak{r}^{ij}\\
&-\Phi\mathfrak{r}^{ir}\mathfrak{r}^{js}(2F'_{km}\mathfrak{r}^{nl}+F''_{kl;mn})\bar{\nabla}_r\mathfrak{r}_{kl}
 \bar{\nabla}_s\mathfrak{r}_{mn}\\
 &-\Phi\mathfrak{r}^{ip}\mathfrak{r}^{jq}(F'_{pk}\mathfrak{r}_{qk}-F'_{qk}\mathfrak{r}_{pk})\\
 &-\mathfrak{r}^{ki}\mathfrak{r}^{lj}\left(\bar{\nabla}_k \Phi\bar{\nabla}_l F+\bar{\nabla}_l\Phi\bar{\nabla}_k F+ F\bar{\nabla}_k\bar{\nabla}_l \Phi\right).
\end{align*}
Furthermore, assume that the maximum eigenvalue of $[\mathfrak{r}^{ij}]_{1\leq i,j\leq n-1}$ is $\mathfrak{r}^{11}=\frac{1}{\lambda_1}$. We have
\begin{align}\label{e: upper P}
 \partial_t \mathfrak{r}^{11}=&\Phi F'_{kl}\bar{\nabla}_k\bar{\nabla}_l \mathfrak{r}^{11}-\Phi(F+F'_{kl}\mathfrak{r}_{kl})(\mathfrak{r}^{11})^2
+\Phi(tr F')\mathfrak{r}^{11}\\
&-\Phi(\mathfrak{r}^{11})^2(2F'_{km}\mathfrak{r}^{nl}+F''_{kl;mn})\bar{\nabla}_1\mathfrak{r}_{kl}\bar{\nabla}_1\mathfrak{r}_{mn}\nonumber\\
&-(\mathfrak{r}^{11})^2\left(2\bar{\nabla}_1\Phi\bar{\nabla}_1 F+ F\bar{\nabla}_1\bar{\nabla}_1 \Phi\right).
\nonumber
\end{align}
\begin{lemma}[Upper bounds on the principal curvatures]\label{lem: upper P}
Assume that $n>2$ and $\Phi\equiv 1.$ Let $\{K_t\}_{[0,t_0]}$ be a smooth, strictly convex solution of equation (\ref{e: flow1}) with
\[C_1\leq F\leq C_2\]
for all $t\in[0,t_0].$
Then
\[\mathfrak{r}^{11} \leq \frac{1}{C_1t},\]
for all $t>0.$
\end{lemma}
\begin{proof}
We estimate the terms on each line of (\ref{e: upper P}).
\begin{description}
\item[a] \emph{Estimating the terms on the first line:} The first term is an essential good term regarded as an elliptic operator which is non-positive at the point and direction where a maximum eigenvalue occurs. The sum of the second and the third term is less than equal to $-(\mathfrak{r}^{11})^2\Phi F  $.
\item[b]\emph{Estimating the terms on the second line:} Notice that by the fourth property of $G$, we have $F([\mathfrak{r}^{ij}]_{1\leq i,j\leq n-1})=\frac{1}{F([\mathfrak{r}_{ij}]_{1\leq i,j\leq n-1})}$. Therefore, by \cite{JU} page 112, this term can be estimated from above by
  $-2\Phi(\mathfrak{r}^{11})^2 \frac{\left(\bar{\nabla}_1F\right)^2}{F}.$
\item[c]\emph{Estimating the terms on the third line:} By the arithmetic mean-geometric mean inequality we have
\[2|\bar{\nabla}_1\Phi\bar{\nabla}_1 F|\leq 2\Phi\frac{(\bar{\nabla}_1F)^2}{F}+F\frac{(\bar{\nabla}_1\Phi)^2}{2\Phi}.\]
\end{description}
Thus, combining these estimates all together we have
\[\partial_t \mathfrak{r}^{11}\leq -(\mathfrak{r}^{11})^2F\left(\Phi+\bar{\nabla}_1\bar{\nabla}_1 \Phi-\frac{(\bar{\nabla}_1\Phi)^2}{2\Phi}\right).\]
Consequently, by the assumption  $\Phi\equiv 1$, we have
\[\partial_t \mathfrak{r}^{11}\leq-C_1(\mathfrak{r}^{11})^2.\]
 An ODE comparison completes the proof.
\end{proof}
Denote the principal curvatures by $\kappa_i=\frac{1}{\lambda_i}$ for $\displaystyle 1\leq i\leq n-1$.
Now we have all necessary prerequisites to obtain lower and upper bounds on the principal curvatures.
\begin{lemma}[Lower and upper bounds on the principal curvatures]\label{lem: lower and upper P}
Assume that $n>2$ and $\Phi\equiv 1.$ Let $\{K_t\}_{[0,t_0]}$ be a smooth, strictly convex solution of equation (\ref{e: flow1}) with
\[C_1\leq F\leq C_2\]
for all $t\in[0,t_0].$
Then there exists constants $C$ and $C'$ depending on $C_1$ and $C_2$ such that, $\forall t\in[0,t_0]$,
\[C' t^{n-2}\leq\kappa_i\leq \frac{C}{t}.\]
\end{lemma}

\begin{proof}
In fact, the upper bound on the principal curvatures has been established in Lemma \ref{lem: upper P}. So, we need to verify the claimed lower bound on the principal curvatures. As $\kappa_1 \geq \kappa_2 \geq \ldots \geq \kappa_{n-1}$, from the assumption and Lemma \ref{lem: upper P} we get
$$C_2^{-\frac{n-1}{p}} \leq F^{-\frac{n-1}{p}}=\mathcal{K}= \prod_{i=1}^{n-1} \kappa_i = \kappa_{n-1} \cdot \prod_{i=1}^{n-2} \kappa_i \leq \kappa_{n-1} \left(\frac{1}{C_1t}\right)^{n-2}.$$
\end{proof}

\section{Proof of the main theorem}
First, we remark that $\lim\limits_{t\to T}\max\limits_{\mathbb{S}^{n-1}}s=\infty.$ This result follows from Lemmas \ref{lem: upper G}, \ref{lem: lower G}, and \ref{lem: lower and upper P} combined with a standard argument that has been used in Corollary I1.9 of \cite{BA}.
Second, we point out that the lifespan of a solution to the inverse Gauss curvature flow (\ref{e: flow1}) is $[0,\infty),$ for $0<p< 1$ and $\Phi\equiv 1$. Indeed, if $K_0\subseteq \mathbb{B}_{R}$, for $R$ large enough, then from the comparison principle it follows that $K_t\subseteq\mathbb{B}_{R(t)}$, where $R(t):=\left(\max\limits_{\mathbb{S}^{n-1}}s(\cdot,0)^{1-p}+(1-p)t\right)^{\frac{1}{1-p}}.$ Thus for a finite time $T,$ the family of support functions $\{s(\cdot,t)\}_{t<T}$ is uniformly bounded by $R(T)$ and this violates $\lim\limits_{t\to T}\max\limits_{\mathbb{S}^{n-1}}s=\infty.$
Additionally, we mention that $\lim\limits_{t\to T}\max\limits_{\mathbb{S}^{n-1}}s=\infty$ and Proposition \ref{prop: chow} together imply that $\lim\limits_{t\to\infty}\min\limits_{\mathbb{S}^{n-1}}s=\infty.$
Therefore, we can use Corollaries \ref{cor: 1} and \ref{cor: 2} in the subsequent argument.

Next, observe that a solution of equation (\ref{e: flow1}) has the following rescaling property.
Let $s:\mathbb{S}^{n-1}\times[0,\infty)\to\mathbb{R}^{n}$ be a solution of equation (\ref{e: flow1}). Then for each $a>0$, $s_{a}$ defined as $s_{a}:\mathbb{S}^{n-1}\times\left[0,\infty\right)\to\mathbb{R}^{n}$ and
\[s_{a}(\theta,t)=a^{\frac{1}{n}} s\left(\theta,a^{\frac{p-1}{n}}t\right)\]
is also a solution of evolution equation (\ref{e: flow1}).

For each \emph{fixed} time $t\in[0,\infty),$ define $\bar{s}$ a solution of (\ref{e: flow1}), by the rescaling property, as follows
\[\bar{s}(z,\tau)=\left(\frac{V(\mathbb{B})}{V(K_t)}\right)^{\frac{1}{n}}s\left(z,
t+\left(\frac{V(\mathbb{B})}{V(K_t)}\right)^{\frac{p-1}{n}}\tau\right).\]
So, $\bar{s}(\cdot,0)$ is the support function of $\left(\frac{V(\mathbb{B})}{V(K_t)}\right)^{\frac{1}{n}}K_t$. Therefore,
\[c\leq \bar{s}(z,0)\leq C,\]
for constants $c$ and $C$ obtained from Corollary \ref{cor: 2}.
Let $\mathbb{B}_C$ denote the ball of radius $C$ centered at the origin. Thus, $\mathbb{B}_{C}$ encloses
the convex body associated with the support function $\bar{s}(\cdot,0).$ The containment principal insures that $\mathbb{B}_{2C}$
will encompass the convex body associated with the support function $\bar{s}(\cdot,\tau),$ for $\tau\in[0,\delta],$
where $\delta:=\frac{1}{1-p}(2^{1-p}-1)C^{1-p}$ is the elapsed time that it take for $\mathbb{B}_{C}$ to expand to $\mathbb{B}_{2C}$ under the flow. Now, Lemmas \ref{lem: upper G}, \ref{lem: lower G}, and \ref{lem: lower and upper P} imply that there are uniform lower and upper bounds on the principal curvatures and on the speed of the flow on the time interval $[\delta/2,\delta].$ Hence, by \cite{K}, we conclude that there are uniform bounds on higher derivatives of the curvature. Accordingly, all quantities related to the original solution that are scaling invariant fulfil uniform bounds on the time interval
$\left[\left(\frac{V(\mathbb{B})}{V(K_0)}\right)^{\frac{p-1}{n}}\frac{\delta}{2},\infty\right):$ Observe that for a fixed time $t$
\[\bar{s}(\cdot,\tau)=\left(\frac{V(\mathbb{B})}{V(K_t)}\right)^{\frac{1}{n}}s\left(\cdot,
t+\left(\frac{V(\mathbb{B})}{V(K_t)}\right)^{\frac{p-1}{n}}\tau\right)\]
has uniform $C^k$ bounds for all $\tau\in[\delta/2,\delta]$ and these bounds are independent of $t$ and $\tau$. Furthermore,
the volume of the convex body corresponding to $\bar{s}(\cdot,\tau)$ is bounded from above by $V(\mathbb{B}_{2C}).$
Thus,
\[\displaystyle\frac{V(\mathbb{B})}{V(\mathbb{B}_{2C})}\leq E_t:=\frac{V(K_t)}{V\left(K_{t+\left(\frac{V(\mathbb{B})}{V(K_t)}\right)^{\frac{p-1}{n}}\tau}\right)}\leq 1\]
for all $t\in [0,\infty)$ and $\tau\in[\delta/2,\delta]$. This implies that we have uniform $C^k$ bounds for
\begin{align*}
\left(\frac{V(\mathbb{B})}{V\left(K_{t+\left(\frac{V(\mathbb{B})}{V(K_t)}\right)^{\frac{p-1}{n}}\tau}\right)}\right)^{\frac{1}{n}}s\left(\cdot,
t+\left(\frac{V(\mathbb{B})}{V(K_t)}\right)^{\frac{p-1}{n}}\tau\right)
=E_t\cdot\bar{s}(\cdot,\tau)
\end{align*}
for all $t\in [0,\infty)$ and $\tau\in[\delta/2,\delta]$, in particular for $\tau=\delta/2.$ Next, we show that for every
 $\tilde{t}\in\left[\left(\frac{V(\mathbb{B})}{V(K_0)}\right)^{\frac{p-1}{n}}\frac{\delta}{2},\infty\right)$, we can find $t\in[0,\infty)$ such that
\[\tilde{t}=t+\left(\frac{V(\mathbb{B})}{V(K_t)}\right)^{\frac{p-1}{n}}\frac{\delta}{2}.\]
Define $f(t)=t+\left(\frac{V(\mathbb{B})}{V(K_t)}\right)^{\frac{p-1}{n}}\frac{\delta}{2}-\tilde{t}$ on $[0,\infty)$.
This is a continuous function. We have $f(\infty)=\infty-\tilde{t}>0$. On the other hand, we have
\[f(0)=\left(\frac{V(\mathbb{B})}{V(K_0)}\right)^{\frac{p-1}{n}}\frac{\delta}{2}-\tilde{t}\leq 0.\]
The claim follows. So, we have proved that:\\
For every $\tilde{t}\in\left[\left(\frac{V(\mathbb{B})}{V(K_0)}\right)^{\frac{p-1}{n}}\frac{\delta}{2},\infty\right)$,
there is a $t\in [0,\infty)$ with $\tilde{t}=t+\left(\frac{V(\mathbb{B})}{V(K_t)}\right)^{\frac{p-1}{n}}\frac{\delta}{2}$, such that the family of convex bodies
\[\left\{\tilde{K}_{\tilde{t}}=\left(\frac{V(\mathbb{B})}{V(K_{\tilde{t}})}\right)^{\frac{1}{n}}K_{\tilde{t}}\right\}_{\tilde{t}\in
\left[\left(\frac{V(\mathbb{B})}{V(K_0)}\right)^{\frac{p-1}{n}}\frac{\delta}{2},\infty\right)}\]
has uniform $C^k$ bounds.\\
To finish the proof of the theorem, recall from Corollary \ref{cor: 1} that \[\lim\limits_{t\to\infty}\frac{\max\limits_{\mathbb{S}^{n-1}}s_{\tilde{K}_t}}{\min\limits_{\mathbb{S}^{n-1}}s_{\tilde{K}_t}}=\lim\limits_{t\to\infty}\frac{\max\limits_{\mathbb{S}^{n-1}} s(\cdot,t)}{\min\limits_{\mathbb{S}^{n-1}} s(\cdot,t)}=1.\]
As $\min\limits_{\mathbb{S}^{n-1}}s_{\tilde{K}_t}\leq1\leq\max\limits_{\mathbb{S}^{n-1}}s_{\tilde{K}_t}$, we infer that $\lim\limits_{t\to\infty} s_{\tilde{K}_t}=1.$ Therefore, $\left\{\left(\frac{V(\mathbb{B})}{V(K_t)}\right)^{\frac{1}{n}}K_t\right\}_t$ converges in the $C^{\infty}$ topology to the origin-centered unit ball, as $t$ approaches $\infty$.

\begin{remark}
An alternative approach to prove Theorem \ref{thm: expanding} is to combine Harnack inequality of Y. Li \cite{LI} with displacement bounds, almost in an identical manner as in \cite{BA}.
\end{remark}
\section{Further discussion: Shrinking Gauss curvature flow}
In the preceding sections, using the evolution equation of dual convex bodies we studied long time behavior of the expanding Gauss curvature flows in the case $\Phi\equiv 1.$ Here, we would like to point out that this method can also be used to obtain $C^2$ regularity of solutions under certain powers of Gauss curvature flow to prove the known Theorem \ref{thm:2}, which was proved by Andrews in \cite{BA4}, including the case $p=1$. The contribution of this section is mainly on the uniform lower bound on the Gauss curvature for the normalized solution.

For a given smooth, strictly convex embedding $x_K$, consider a family of smooth convex bodies $\{K_t\}_t$ given by the smooth embeddings $x:\partial K\times[0,T)\to \mathbb{R}^{n}$, which
are evolving according to the anisotropic shrinking Gauss curvature flow, namely,
\begin{equation}\label{e: flow2}
 \partial_{t}x(\cdot,t):=\Phi(\mathbf{n}_{K_t}(\cdot))\mathcal{K}^{\frac{p}{n-1}}\, \mathbf{n}_{K_t}(\cdot),~~
 x(\cdot,0)=x_{K}(\cdot)
\end{equation}
where $\Phi:\mathbb{S}^{n-1}\to\mathbb{R}$ is a positive smooth function, $0<T<\infty$, and $p\in \left(0,\infty\right).$
In this equation $x(\partial K,t)=\partial K_t.$
\begin{theorem}\label{thm:2}\cite{BA4}
Assume that $p\in\left(\frac{n-1}{n+1},1\right).$ Let $x_{K}$ be a smooth, strictly convex embedding of $\partial K$. Then there exists a smooth, unique solution of (\ref{e: flow2}). Furthermore, there is a point $q\in\mathbb{R}^{n}$ such that the rescaled convex bodies given by $\left(\frac{V(\mathbb{B})}{V(K_t)}\right)^{\frac{1}{n}}\left(K_t-q\right)$ converge in the $C^{\infty}$ topology to a homothetic solution of (\ref{e: flow2}). Additionally, if $p\in\left(0,\frac{n-1}{n+1}\right)$ and the isoperimetric ratio remains bounded, then the same conclusion holds as in the case of $p\in\left(\frac{n-1}{n+1},1\right).$
\end{theorem}
Notice that for $\Phi\equiv1$ and $p=\frac{n-1}{n+1}$ the flow (\ref{e: flow2}) is the well-known affine normal flow which has been studied in
\cite{BA4,AST,Ivaki1,IS,ST}.

On the contrary to the statement of Theorem \ref{thm: expanding}, we do not need to confine ourselves to the case $\Phi\equiv1.$ The reason is that Theorem 10 of \cite{BA4} is available for any positive smooth function $\Phi.$
\begin{remark}\label{rem: centroid}
For a convex body $K$, denote the maximum width by $\omega_+(K):=\max\limits_{\mathbb{S}^{n-1}}(s(z)+s(-z))$ and the minimum width by $\omega_{-}(K):=\min\limits_{\mathbb{S}^{n-1}}(s(z)+s(-z))$. A convex body with centroid $b$, satisfies $\frac{\omega_{-}}{n+1}\mathbb{B}\subseteq K-b\subseteq \frac{n\omega_{+}}{n+1}\mathbb{B}$; this result goes back to Minkowski (1897), see Schneider \cite{schneider}, page 308. This in particular implies that $\frac{n+1}{n\omega_{+}}\mathbb{B}\subseteq (K-b)^{\circ}\subseteq \frac{n+1}{\omega_{-}}\mathbb{B}.$
\end{remark}
In connection to the flow (\ref{e: flow2}), after obtaining the evolution equation of the corresponding dual convex bodies, one can acquire a uniform lower bound on the Gauss curvature in terms of inradius, circumradius, and time; likewise done in Lemma 2.3 in \cite{IS}: In \cite{BA4}, it was proved that the ratio $\omega_{+}(K_t)/\omega_{-}(K_t)$ remains bounded along the flow (\ref{e: flow2}), if $p\in\left(\frac{n-1}{n+1},1\right)$. Therefore, the rescaled solution given by $\left\{\tilde{K}_t=\left(\frac{V(\mathbb{B})}{V(K_t)}\right)^{\frac{1}{n}}K_t\right\}_{t\in[0,T)}$ satisfy $\omega_+(\tilde{K}_t)\leq \tilde{C}<\infty$ and $\omega_-(\tilde{K}_t)\geq \tilde{c}>0$. On the other hand, if $s(\cdot,t)$ is the solution to (\ref{e: flow2}) then for each \emph{fixed} time $t$,
\[\bar{s}(z,\tau)=\left(\frac{V(\mathbb{B})}{V(K_t)}\right)^{\frac{1}{n}}\left(s\left(z,
t+\left(\frac{V(\mathbb{B})}{V(K_t)}\right)^{-\frac{p}{n}}\tau\right)-\langle b_t,z\rangle\right)\]
is also a solution to (\ref{e: flow2}) with the initial data $\bar{s}(z,0)=\left(\frac{V(\mathbb{B})}{V(K_t)}\right)^{\frac{1}{n}}\left(s(z,t)-\langle b_t,z\rangle\right),$ here $b_t$ is the centroid of $K_t.$ So, from Remark \ref{rem: centroid} we get $\frac{\tilde{c}}{n+1}\leq \bar{s}(z,0)\leq \frac{n\tilde{C}}{n+1}$ and $\frac{n+1}{n\tilde{C}}\leq \bar{s}^{\circ}(z,0)\leq \frac{n+1}{\tilde{c}}.$ Now we can use the argument of \cite{IS} (similar to the argument in the proof of Lemma \ref{lem: lower G}) to obtain a uniform lower bound on the Gauss curvature (which is invariant under Euclidian translations) for the normalized solution, $\tilde{K}_t$, in terms of $\tilde{c},\tilde{C},n$, and $p$. We can then continue as in \cite{BA4} to prove Theorem \ref{thm:2}. The advantage is that we avoid using the Haranck inequality and the displacement bounds of \cite{BA4}. However, we mention that Haranck inequalities and displacement bounds are fundamental ingredients to prove existence of solutions to some flows with non-smooth initial data, see \cite{BA4,BYM}. Moreover, they can also be used to obtain stability of some geometric inequalities that come from geometric flows \cite{Ivaki}. Therefore, by no means our method here substitutes those in \cite{BA,BA4}, but abridges the regularity arguments. For some powers $p>1$ and $\Phi\equiv1$, the flow (\ref{e: flow2}) has been treated in \cite{BA2,BX,Ch,PL}.\\\\
\textbf{Acknowledgment.}\\
 I would like to thank the referee for helpful comments and the thorough attention given to the manuscript.
\bibliographystyle{amsplain}

\end{document}